\documentclass[a4paper]{article}

\usepackage{hyperref} 

\usepackage{amssymb,amsmath,amsthm}
\usepackage[english]{babel}

\usepackage{todonotes}
\usepackage{verbatim} 
\usepackage{enumerate}
\usepackage{mathtools}

\usepackage{bbm}

\usepackage{cite}

\newcommand{\cA}{\mathcal{A}}

\newcommand{\cC}{\mathcal{C}}
\newcommand{\cD}{\mathcal{D}}

\newcommand{\cL}{\mathcal{L}}
\newcommand{\cM}{\mathcal{M}}
\newcommand{\cN}{\mathcal{N}}

\newcommand{\cP}{\mathcal{P}}

\newcommand{\bE}{\mathbb{E}}

\newcommand{\bN}{\mathbb{N}}

\newcommand{\bR}{\mathbb{R}}

\newcommand{\PR}{\mathbb{P}}

\newcommand{\bONE}{\mathbbm{1}}

\newcommand{\dd}{ \mathrm{d}}

\renewcommand{\epsilon}{\varepsilon}

\newcommand{\vn}[1]{\left| \! \left| #1\right| \! \right|}

\newcommand{\ip}[2]{\langle #1,#2\rangle}

\numberwithin{equation}{section}

\newtheorem{theorem}{Theorem}[section]
\newtheorem{lemma}[theorem]{Lemma}
\newtheorem{proposition}[theorem]{Proposition}

\theoremstyle{definition}
\newtheorem{definition}[theorem]{Definition}

\newtheorem{remark}[theorem]{Remark}

\newtheorem{example}[theorem]{Example}

\begin{document}

%%%%%%%%%%%%%%%%%%%%%%%%%%%%%%%%%%
\title{Large deviations for Markov jump processes with mean-field interaction via the comparison principle for an associated Hamilton-Jacobi equation}

\author{
\renewcommand{\thefootnote}{\arabic{footnote}}
Richard Kraaij
\footnotemark[1]
}

\footnotetext[1]{
Delft Institute of Applied Mathematics, Delft University of Technology, Mekelweg 4, 
2628 CD Delft, The Netherlands, E-mail: \texttt{r.c.kraaij@tudelft.nl}.
}

\maketitle
%%%%%%%%%%%%%%%%%%%%%%%%%%%%%%%%%%

\begin{abstract}
We prove the large deviation principle for the trajectory of a broad class of mean field interacting Markov jump processes via a general analytic approach based on viscosity solutions. Examples include generalized Ehrenfest models as well as Curie-Weiss spin flip dynamics with singular jump rates.

The main step in the proof of the large deviation principle, which is of independent interest, is the proof of the comparison principle for an associated collection of Hamilton-Jacobi equations. 

Additionally, we show that the large deviation principle provides a general method to identify a Lyapunov function for the associated McKean-Vlasov equation.
\end{abstract}

%% Keywords
%% Large deviations, non-linear jump processes, viscosity solutions, comparison principle, Hamilton-Jacobi equation

%% 60F10 large deviations
%% 60J75 jump processes
%% 35D40 viscosity solutions

\section{Introduction}

We consider two models of Markov jump processes with mean field interaction. In both cases, we have $n$ particles or spins that evolve as a pure jump process, where the jump rates of the individual particles depend on the empirical distribution of all $n$ particles.

We prove the large deviation principle(LDP) for the trajectory of these empirical quantities and show that the rate function is in Lagrangian form. The first set of models that we consider are conservative models that generalize the Ehrenfest model. In the one dimensional setting, this model is also known as the Moran model without mutation or selection. For these models, the empirical quantity of interest for large $n$ is the empirical magnetisation. The second class of models are jump processes of Glauber type such as Curie-Weiss spin flip dynamics. In this case, the empirical measure is given by
\begin{equation*}
\mu_n(t) := \frac{1}{n}\sum_{i \leq n} \delta_{\sigma_i(t)},
\end{equation*}
where $\sigma_i(t) \in \{1,\dots,d\}$ is the state of the $i$-th spin at time $t$. Under some appropriate conditions, the trajectory $\mu_n(t)$ converges as $n \rightarrow \infty$ to $\mu(t)$, the solution of a McKean-Vlasov equation, which is a generalization of the linear Kolmogorov forward equation which would appear in the case of independent particles. 

For the second class of models, we obtain a large deviation principle for the trajectory of these empirical measures on the space $D_{\cP(\{1,\dots,d\})}(\bR^+)$ of c\`{a}dl\`{a}g paths on $E := \cP(\{1,\dots,d\})$ of the form
\begin{equation*}
\PR\left[\{\mu_n(t)\}_{t \geq 0} \approx \gamma\right] \approx e^{-nI(\gamma)}
\end{equation*}
where 
\begin{equation*}
I(\gamma) = I(\gamma(0)) + \int_0^\infty \cL(\gamma(s),\dot{\gamma}(s)) \dd s
\end{equation*}
for trajectories $\gamma$ that are absolutely continuous and $I(\gamma) = \infty$ otherwise. In particular, $I(\gamma) = 0$ for the solution $\gamma$ of the limiting McKean-Vlasov equation. The \textit{Lagrangian} $\cL : E \times \bR^d \rightarrow \bR^+$ is defined as the Legendre transform of a \textit{Hamiltionan} $H : E \times \bR^d \rightarrow \bR$ that can be obtained via a limiting procedure 
\begin{equation} \label{eqn:convergence_of_H_intro}
H(x,\nabla f(x)) = Hf(x) = \lim_n \frac{1}{n} e^{nf} A_n e^{nf}.
\end{equation}
Here $A_n$ is the generator of the Markov process of $\{\mu_n(t)\}_{t \geq 0}$. More details on the models and definitions follow shortly in Section \ref{section:main_results}.

Recent applications of the path-space large deviation principle are found in the study of mean-field Gibbs-non-Gibbs transitions, see e.g. \cite{EK10,EFHR10} or the microscopic origin of gradient flow structures, see e.g. \cite{AdDiPeZi13,MPR13}. Other authors have considered the path-space LDP in various contexts before, see for example \cite{FW98,Co89,Le95,DPdH96,Fe94b,BuDuMa11,BoSu12}. A comparison with these results follows in Section \ref{section:comparison_to_results_in_literature}.

The novel aspect of this paper with respect to large deviations for jump processes is an approach via a class of \textit{Hamilton-Jacobi} equations. In \cite{FK06}, a general strategy is proposed for the study for large deviations of trajectories which is based on the convergence of non-linear semigroups. As in the theory of weak convergence of Markov processes, this program is carried out in two steps, first one proves convergence of the generators, i.e. \eqref{eqn:convergence_of_H_intro}, and secondly one shows that $H$ is indeed the generator of a semigroup.

The latter issue is non trivial and follows for example by showing that the Hamilton-Jacobi equation
\begin{equation} \label{eqn:Hamilton_Jacobi_equation_intro}
f(x) - \lambda H(x,\nabla f(x)) - h(x) = 0
\end{equation}
has a unique solution $f$ for all $h \in C(E)$ and $\lambda >0$ in the viscosity sense. It is exactly this problem that is the main focus of the paper. An extra bonus of this approach is that the conditions on the Markov processes for finite $N$ are weaker then in previous studies, and allow for singular behaviour in the jump rate for a particle to move from $a$ to $b$ in boundary regions when the empirical average $\mu(a)$ is close to $0$.

This approach via the Hamilton-Jacobi equation has been carried out in \cite{FK06} for Levy processes on $\bR^d$, systems with multiple time scales and for stochastic equations in infinite dimensions. In \cite{DFL11}, the LDP for a diffusion process on $(0,\infty)$ is treated with singular behaviour close to $0$.

\smallskip

As a direct consequence of our large deviation principle, we obtain a straightforward method to find Lyapunov functions for the limiting McKean-Vlasov equation. If $A_n$ is the linear generator of the empirical quantity of interest of the $n$-particle process, the operator $A$ obtained by $Af = \lim_n A_n f$ can be represented by $Af(\mu) = \ip{\nabla f(\mu)}{\mathbf{F}(\mu)}$ for some vector field $\mathbf{F}$. If solutions to 
\begin{equation} \label{eqn:vector_field_equation_intro}
\dot{\mu}(t) = \mathbf{F}(\mu(t))
\end{equation}
are unique for a given starting point and if the empirical measures $\mu_n(0)$ converges to $\mu(0)$, the empirical measures $\{\mu_n(t)\}_{t \geq 0}$ converge almost surely to a solution $\{\mu(t)\}_{t \geq 0}$ of \eqref{eqn:vector_field_equation_intro}. In Section \ref{subsection:Htheorems}, we will show that if the stationary measures of $A_n$ satisfy a large deviation principle on $\cP(\{1,\dots,d\})$ with rate function $I_0$, then $I_0$ is a Lyapunov function for \eqref{eqn:vector_field_equation_intro}. 

\smallskip

The paper is organised as follows. In Section \ref{section:main_results}, we introduce the models and state our results. Additionally, we  give some examples to show how to apply the theorems. In Section \ref{section:LDP_via_HJequation}, we recall the main results from \cite{FK06} that relate the Hamilton-Jacobi equations \eqref{eqn:Hamilton_Jacobi_equation_intro} to the large deviation problem. Additionally, we verify conditions from \cite{FK06} that are necessary to obtain our large deviation result wit a rate function in Lagrangian form, in the case that we have uniqueness of solutions to the Hamilton-Jacobi equations. Finally, in Section \ref{section:viscosity_solutions} we prove uniqueness of viscosity solutions to \eqref{eqn:Hamilton_Jacobi_equation_intro}.

\section{Main results} \label{section:main_results}

\subsection{Two models of interacting jump processes} \label{section:two_models}

We do a large deviation analysis of the trajectory of the empirical magnetization or distribution for two models of interacting spin-flip systems. The first setting is a $d$-dimensional Ehrenfest model.

\smallskip

\textbf{Generalized Ehrenfest model in $d$-dimensions.}

Consider $d$-dimensional spins $\sigma = (\sigma(1),\dots,\sigma(n)) \in (\{-1,1\}^d)^n$. For example, we can interpret this as $n$ individuals with $d$ types, either being $-1$ or $1$. For $k \leq n$, we denote the $i$-th coordinate of $\sigma(k)$ by $\sigma_i(k)$. Set $x_n = (x_{n,1},\dots,x_{n,d}) \in E_1 := [-1,1]^d$, where $x_{n,i} = x_{n,i}(\sigma) = \frac{1}{n} \sum_{j=1}^n \sigma_i(j)$ the empirical magnetisation in the $i$-th spin. For later convenience, denote by $E_{1,n}$ the discrete subspace of $E_1$ which is the image of $(\{-1,1\}^d)^n$ under the map $\sigma \mapsto x_n(\sigma)$. The spins evolve according to mean-field Markovian dynamics with generator $\cA_n$:
\begin{multline*}
\cA_n f(\sigma) = \sum_{i = 1}^d \sum_{j = 1}^n \bONE_{\{\sigma_i(j) = -1\}} r_{n,+}^i(x_n(\sigma)) \left[f(\sigma^{i,j}) - f(\sigma)\right] \\
+ \sum_{i = 1}^d \sum_{j = 1}^n \bONE_{\{\sigma_i(j) = 1\}} r_{n,-}^i(x_n(\sigma)) \left[f(\sigma^{i,j}) - f(\sigma)\right].
\end{multline*}
The configuration $\sigma^{i,j}$ is obtained by flipping the $i$-th coordinate of the $j$-th spin. The functions $r_{n,+}^i,r_{n,-}^i$ are non-negative and represent the jump rate of the $i$-th spin flipping from a $-1$ to $1$ or vice-versa.

The empirical magnetisation $x_n$ itself also behaves Markovian and has generator
\begin{multline*}
A_n f(x) = \sum_{i = 1}^d \Bigg\{ n \frac{1-x_i}{2}r_{n,+}^i(x) \left[f\left(x + \frac{2}{n}e_i\right) - f(x) \right] \\
+ n \frac{1+x_i}{2}r_{n,-}^i(x)\left[f\left(x - \frac{2}{n}e_i\right) - f(x) \right] \Bigg\},
\end{multline*}
where $e_i$ the vector consisting of $0$'s, and a $1$ in the $i$-th component.

Under suitable conditions on the rates $r_{n,+}^i$ and $r_{n,-}^i$, we will derive a large deviation principle for the trajectory $\{x_n(t)\}_{t \geq 0}$ in the Skorokhod space $D_{E_1}(\bR^+)$ of right continuous $E_1$ valued paths that have left limits. 

\smallskip

\textbf{Systems of Glauber type with $d$ states.}

We will also study the large deviation behaviour of copies of a Markov process on $\{1,\dots,d\}$ that evolve under the influence of some mean-field interaction. Here $\sigma = (\sigma(1),\dots,\sigma(n)) \in \{1,\dots,d\}^n$ and the empirical distribution $\mu$ is given by $\mu_n(\sigma) = \frac{1}{n} \sum_{i \leq n} \delta_{\sigma(i)}$ which takes values in
\begin{equation*}
E_{2,n} := \left\{\mu \in \cP(E) \, \middle| \, \mu = \frac{1}{n} \sum_{i=1}^n \delta_{x_i}, \text{ for some } x_i \in \{1,\dots,d\} \right\}.
\end{equation*}
Of course, this set can be seen as discrete subset of $E_2 := \cP(\{1,\dots,d\}) = \{\mu \in \bR^d \, | \, \mu_i \geq 0, \sum_i \mu_i = 1\}$. We take some $n$-dependent family of jump kernels $r_n : \{1,\dots, d\}\times\{1,\dots, d\}\times E_n \rightarrow \bR^+$ and define Markovian evolutions for $\sigma$ by
\begin{equation*}
\cA_nf(\sigma(1),\dots,\sigma(n)) = \sum_{i=1}^n \sum_{b = 1}^d r_n\left(\sigma(i),b, \frac{1}{n} \sum_{i=1}^n \delta_{\sigma(i)}\right) \left[f(\sigma^{i,b}) - f(\sigma) \right],
\end{equation*}
where $\sigma^{i,b}$ is the configuration obtained from $\sigma$ by changing the $i$-th coordinate to $b$. Again, we have an effective evolution for $\mu_n$, which is governed by the generator
\begin{equation*}
A_n f(\mu) = n \sum_{a,b} \mu(a)r_n(a,b,\mu)\left[f\left(\mu - n^{-1}\delta_a + n^{-1} \delta_b\right) - f(\mu)\right].
\end{equation*}
As in the first model, we will prove, under suitable conditions on the jump kernels $r_n$ a large deviation principle in $n$ for $\{\mu_n(t)\}_{t \geq 0}$ in the Skorokhod space $D_{E_2}(\bR^+)$.

\subsection{Large deviation principles} \label{section:main_ldp}

The main results in this paper are the two large deviation principles for the two sets of models introduced above. To be precise, we say that the sequence $x_n \in D_{E_1}(\bR^+)$, or for the second case $\mu_n \in D_{E_2}(\bR^+)$, satisfies the large deviation principle with rate function $I : D_{E_1}(\bR^+) \rightarrow [0,\infty]$ if $I$ is lower semi-continuous and the following two inequalities hold:
\begin{enumerate}[(a)]
\item For all closed sets $G \subseteq D_{E_1}(\bR^+)$, we have
\begin{equation*}
\limsup_{n \rightarrow \infty} \frac{1}{n} \log \PR[\{x_n(t)\}_{t \geq 0} \in G] \leq - \inf_{\gamma \in G} I(\gamma).
\end{equation*}
\item For all open sets $U \subseteq D_{E_1}(\bR^+)$, we have
\begin{equation*}
\liminf_{n \rightarrow \infty} \frac{1}{n} \log \PR[\{x_n(t)\}_{t \geq 0} \in U] \geq - \inf_{\gamma \in G} I(\gamma).
\end{equation*}
\end{enumerate}
For the definition of the Skorokhod topology defined on $D_{E_1}(\bR^+)$, see for example \cite{EK86}. We say that $I$ is \textit{good} if the level sets $I^{-1}[0,a]$ are compact for all $a \geq 0$. 

\smallskip

For a trajectory $\gamma \in D_{E_1}(\bR)$, we say that $\gamma \in \cA\cC$ if the trajectory is absolutely continuous. For the $d$-dimensional Ehrenfest model, we have the following result.

\begin{theorem} \label{theorem:ldp_ehrenfest_dd}
Suppose that there exists a family of continuous functions $v_+^i, v_-^i : E_1 \rightarrow \bR^+$, $1 \leq i \leq d$, such that
\begin{equation} \label{eqn:main_convergence_condition_for_ldp_ehrenfest}
\lim_{n \rightarrow \infty} \sup_{x \in E_n} \sum_{i=1}^d \left|\frac{1-x_i}{2}r_{n,+}^i(x) - v_+^i(x)\right| + \left|\frac{1+x_i}{2}r_{n,-}^i(x) - v_-^i(x)\right|= 0.
\end{equation}
Suppose that for every $i$, the functions $v_+^i$ and $v_-^i$ satisfy the following.

The rate $v_+^i$ is identically zero or we have the following set of conditions.
\begin{enumerate}[(a)]
\item $v_+^i(x) > 0$ if $x_i \neq 1$.
\item For $z \in [-1,1]^d$ such that $z_i = 1$, we have $v_+^i(z) = 0$ and for every such $z$ there exists a neighbourhood $U_z$ of $z$ on which there exists a decomposition $v_+^i(x) = v_{+,z,\dagger}^i(x_i) v_{+,z,\ddagger}^i(x)$, where $v_{+,z,\dagger}^i$ is decreasing and where $v_{+,z,\ddagger}^i$ is continuous and satisfies $v_{+,z,\ddagger}^i(z) \neq 0$.
\end{enumerate}
The rate $v_-^i$ is identically zero or we have the following set of conditions.
\begin{enumerate}[(a)]
\item $v_-^i(x) > 0$ if $x_i \neq -1$.
\item For $z \in [-1,1]^d$ such that $z_i = -1$, we have $v_-^i(z) = 0$ and for every such $z$ there exists a neighbourhood $U_z$ of $z$ on which there exists a decomposition $v_-^i(x) = v_{-,z,\dagger}^i(x_i) v_{-,z,\ddagger}^i(x)$, where $v_{+,z,\dagger}^i$ is increasing and where $v_{-,z,\ddagger}^i$ is continuous and satisfies $v_{-,z,\ddagger}^i(z) \neq 0$.
\end{enumerate}

Furthermore, suppose that $\{x_n(0)\}_{n \geq 1}$ satisfies the large deviation principle on $E_1$ with good rate function $I_0$. Then, $\{x_n\}_{n \geq 1}$ satisfies the large deviation principle on $D_{E_1}(\bR^+)$ with good rate function $I$ given by
\begin{equation*}
I(\gamma) = 
\begin{cases}
I_0(\gamma(0)) + \int_0^\infty \cL(\gamma(s),\dot{\gamma}(s)) \dd s & \text{if } \gamma \in \cA\cC, \\
\infty & \text{otherwise}
\end{cases}
\end{equation*}
where the \textit{Lagrangian} $\cL(x,v) : E_1 \times \bR^d \rightarrow \bR$ is given by the Legendre transform $\cL(x,v) = \sup_{p \in \bR^d} \ip{p}{v} - H(x,p)$ of the \textit{Hamiltonian} $H: E_1 \times \bR^d \rightarrow \bR$, defined by
\begin{equation} \label{eqn:def_Hamiltionian_Ehrenfest}
H(x,p) = \sum_{i = 1}^d v_+^i(x) \left[e^{2p_i} - 1 \right] + v_-^i(x) \left[e^{-2p_i} - 1 \right].
\end{equation}
\end{theorem}

\begin{remark} \label{remark:singular_rates}
Note that the functions $v_+^i$ and $v_-^i$ do not have to be of the form $v_+^i(x) = \frac{1-x_i}{2} r_+^i(x)$ for some bounded function $r_+^i$. This we call singular behaviour, as such a rate cannot be obtained the large deviation principle for independent particles, Varadhan's lemma and the contraction principle as in \cite{Le95} or \cite{DPdH96}.
\end{remark}

\begin{theorem} \label{theorem:ldp_mean_field_jump_process}
Suppose there exists a continuous function $v : \{1,\dots,d\}\times\{1,\dots,d\} \times E_2 \rightarrow \bR^+$ such that for all $a,b \in \{1,\dots, d\}$, we have
\begin{equation} \label{eqn:property_of_jump_kernels}
\lim_{n \rightarrow \infty} \sup_{\mu \in E_n} \left|\mu(x) r_n(a,b,\mu) - v(a,b,\eta_n(\mu))\right| = 0.
\end{equation}
Suppose that for each $a,b$, the map $\mu \mapsto v(a,b,\mu)$ is either identically equal to zero or satisfies the following two properties.
\begin{enumerate}[(a)]
\item $v(a,b,\mu) > 0$ for all $\mu$ such that $\mu(a) > 0$. 
\item For $\nu$ such that $\nu(a) = 0$, there exists a neighbourhood $U_\nu$ of $\nu$ on which there exists a decomposition $v(a,b,\mu) = v_{\nu,\dagger}(a,b,\mu(a)) v_{\nu,\ddagger}(a,b,\mu)$ such that $v_{\nu,\dagger}$ is increasing in the third coordinate and such that $v_{\nu,\ddagger}(a,b,\cdot)$ is continuous and satisfies $v_{\nu,\ddagger}(a,b,\nu) \neq 0$.
\end{enumerate}
Additionally, suppose that $\{\mu_n(0)\}_{n \geq 1}$ satisfies the large deviation principle on $E_2$ with good rate function $I_0$. Then, $\{\mu_n\}_{n \geq 1}$ satisfies the large deviation principle on $D_{E_2}(\bR^+)$ with good rate function $I$ given by
\begin{equation*}
I(\gamma) = 
\begin{cases}
I_0(\gamma(0)) + \int_0^\infty \cL(\gamma(s),\dot{\gamma}(s)) \dd s & \text{if } \gamma \in \cA\cC \\
\infty & \text{otherwise},
\end{cases}
\end{equation*}
where $\cL : E_2 \times \bR^d \rightarrow \bR^+$ is the Legendre transform of $H : E_2 \times \bR^d \rightarrow \bR$ given by
\begin{equation} \label{eqn:def_Hamiltionian_jump}
H(\mu,p) =  \sum_{a,b} v(a,b,\mu)\left[e^{p_b - p_a} -1\right].
\end{equation}
\end{theorem}

\subsection{The comparison principle} \label{section:comparison_priniple_main_results}

The main results in this paper are the two large deviation principles as stated above. However, the main step in the proof of these principles is the verification of the comparison principle for a set of Hamilton-Jacobi equations. As this result is of independent interest, we state these results here as well, and leave explanation on why these equation are relevant for the large deviation principles for later. We start with some definitions.

For $E$ equals $E_1$ or $E_2$, let $H : E \times \bR^d \rightarrow \bR$ be some continuous map. For $\lambda > 0$ and $h \in C(E)$ Set $F_{\lambda,h} : E \times \bR \times \bR^d \rightarrow \bR$ by 
\begin{equation*}
F_{\lambda,h}(x,a,p) = a - \lambda H(x,p) - h(x).
\end{equation*}
We will solve the \textit{Hamilton-Jacobi} equation
\begin{equation} \label{eqn:differential_equation_intro}
F_{\lambda,h}(x,f(x),\nabla f(x)) = f(x) - \lambda H(x, \nabla f(x)) - h(x) = 0 \qquad x \in E,
\end{equation}
in the \textit{viscosity} sense.

\begin{definition} \label{definition:viscosity}
We say that $u$ is a \textit{(viscosity) subsolution} of equation \eqref{eqn:differential_equation_intro} if $u$ is bounded, upper semi-continuous and if for every $f \in C^{1}(E)$ and $x \in E$ such that $u - f$ has a maximum at $x$, we have
\begin{equation*}
F_{\lambda,h}(x,u(x),\nabla f(x)) \leq 0.
\end{equation*}
We say that $u$ is a \textit{(viscosity) supersolution} of equation \eqref{eqn:differential_equation_intro} if $u$ is bounded, lower semi-continuous and if for every $f \in C^{1}(E)$ and $x \in E$ such that $u - f$ has a minimum at $x$, we have
\begin{equation*}
F_{\lambda,h}(x,u(x),\nabla f(x)) \geq 0.
\end{equation*}
We say that $u$ is a \textit{(viscosity) solution} of equation \eqref{eqn:differential_equation_intro} if it is both a sub and a super solution.
\end{definition}

\begin{definition} 
We say that equation \eqref{eqn:differential_equation_intro} satisfies the \textit{comparison principle} if for a subsolution $u$ and supersolution $v$ we have $u \leq v$.
\end{definition}

Note that if the comparison principle is satisfied, then a viscosity solution is unique.

\begin{theorem} \label{theorem:comparison_ehrenfest_ddim}
Suppose that $H : E_1 \times \bR^d \rightarrow \bR$ is given by \eqref{theorem:ldp_ehrenfest_dd} and that the family of functions $v_+^i, v_-^i : E_1 \rightarrow \bR^+$, $1 \leq i \leq d$, satisfy the conditions of Theorem \ref{theorem:ldp_ehrenfest_dd}.

Then, for every $\lambda > 0$ and $h \in C(E_1)$, the comparison principle holds for $f(x) - \lambda H(x,\nabla f(x)) - h(x) = 0$.
\end{theorem}

\begin{theorem}\label{theorem:comparison_mean_field_jump} 
Suppose that $H : E_2 \times \bR^d \rightarrow \bR$ is given by \eqref{theorem:ldp_ehrenfest_dd} and that function  $v : \{1,\dots,d\}\times\{1,\dots,d\} \times E_2 \rightarrow \bR^+$ satisfies the conditions of Theorem \ref{theorem:ldp_mean_field_jump_process}.

Then, for every $\lambda > 0$ and $h \in C(E_2)$, the comparison principle holds for $f(\mu) - \lambda H(\mu,\nabla f(\mu)) - h(\mu) = 0$.
\end{theorem}

The main consequence of the comparison principle for the Hamilton-Jacobi equations stems from the fact, as we will see below, that the operator $H$ generates a strongly continuous contraction semigroup on $C(E)$. 

\smallskip

The proof of the large deviation principle is, in a sense, a problem of semigroup convergence. At least for linear semigroups, it is well known that semigroup convergence can be proven via the convergence of their generators. The main issue in this approach is to prove that the limiting generator $H$ generates a semigroup. It is exactly this issue that the comparison principle takes care of.

Hence, the independent interest of the comparison principle comes from the fact that we have semigroup convergence whatever the approximating semigroups are, as long as their generators converge to $H$, i.e. this holds not just for the specifically chosen approximating semigroups that we consider in Section \ref{section:LDP_via_HJequation}.

\subsection{A Lyapunov function for the limiting dynamics} \label{subsection:Htheorems}

As a corollary to the large deviation results, we show how to obtain a Lyapunov function for the solutions of
\begin{equation} \label{eqn:limiting_differential_equation}
\dot{x}(t) = \mathbf{F}(x),
\end{equation}
where $\mathbf{F}(x) := H_p(x,0)$ for a Hamiltonian as in \eqref{eqn:def_Hamiltionian_jump} or \eqref{eqn:def_Hamiltionian_Ehrenfest}. Here $H_p(x,p)$ is interpreted as the vector of partial derivatives of $H$ in the second coordinate.

We will see in Example \ref{example:ehrenfest_non_unique_solution} that the trajectories that solve this differential equation are the trajectories with $0$ cost. Additionally, the limiting operator $(A,C^1(E))$ obtained by
\begin{equation*}
\sup_{x \in E_n \cap K} |A_nf(x) - Af(x)| \rightarrow 0
\end{equation*}
for all $f \in C^1(E)$ and compact sets $K \subseteq E$ has the form by $Af(x) = \ip{\nabla f(x)}{\mathbf{F}(x)}$ for the same vector field $\mathbf{F}$. This implies that the $0$-cost trajectories are solutions to the McKean-Vlasov equation \eqref{eqn:limiting_differential_equation}. Solutions to \ref{eqn:limiting_differential_equation} are not necessarily unique, see Example \ref{example:ehrenfest_non_unique_solution}. Uniqueness holds for example under a one-sided Lipschitz condition: if there exists $M > 0 $ such that $\ip{\mathbf{F}(x) - \mathbf{F}(y)}{x-y} \leq M |x-y|^2$ for all $x,y \in E$.

\smallskip

For non-interacting systems, it is well known that the relative entropy with respect to the stationary measure is a Lyapunov function for solutions of \eqref{eqn:limiting_differential_equation}. The large deviation principle explains this fact and gives a method to obtain a suitable Lyapunov function, also for interacting dynamics.

\begin{proposition} \label{proposition:Htheorem}
Suppose the conditions for Theorem \ref{theorem:ldp_ehrenfest_dd} or Theorem \ref{theorem:ldp_mean_field_jump_process} are satisfied. Suppose there exists measures $\nu_n \in \cP(E_n) \subseteq \cP(E)$ that are invariant for the dynamics generated by $A_n$. Furthermore, suppose that the measures $\nu_n$ satisfy the large deviation principle on $E$ with good rate function $I_0$.

Then $I_0$ is increasing along any solution of $\dot{x}(t) = \mathbf{F}(x(t))$. 
\end{proposition}

Note that we do not assume that  \eqref{eqn:limiting_differential_equation} has a unique solution for a given starting point.

\subsection{Examples} \label{section:examples_main_results}

We give a series of examples to show the extend of Theorems \ref{theorem:ldp_ehrenfest_dd} and \ref{theorem:ldp_mean_field_jump_process}.

For the Ehrenfest model, we start with the basic case, of spins flipping under the influence of some mean field potential. 
\begin{example} \label{example:Ehrenfest_with_potential}
To be precise, fix some continuously differentiable $V : [-1,1]^d \rightarrow \bR$ and set for every $n \geq 1$ and $i \in \{1,\dots,d\}$ the rates
\begin{align*}
r_{n,i}^+(x) & = \exp\left\{- n 2^{-1} \left(V\left(x + \frac{2}{n} e_i\right) - V(x) \right)\right\}, \\
r_{n,i}^-(x) & = \exp\left\{- n 2^{-1} \left(V\left(x - \frac{2}{n} e_i\right) - V(x) \right)\right\}.
\end{align*}
The limiting objects $v_+^i$ and $v_-^i$ are given by
\begin{equation*}
v_+^i(x) = \frac{1-x_i}{2} e^{-\nabla_i V(x)}, \qquad v_-^i(x) = \frac{1+x_i}{2} e^{\nabla_i V(x)},
\end{equation*}
which already have the decomposition as required in the conditions of the Theorem \ref{theorem:ldp_ehrenfest_dd}. For example, condition (b) for $v_+^i$ is satisfied by
\begin{equation*}
v_{+,z,\dagger}^i(x_i) := \frac{1-x_i}{2}, \qquad v_{+,z,\ddagger}^i(x) := e^{-\nabla_i V(x)}.
\end{equation*}
\end{example}

For $d = 1$, we give two extra notable examples, the first one exhibits unbounded jump rates for the individual spins if the empirical magnetisation is close to one of the boundary points. The second example shows a case where we have multiple trajectories $\gamma$ with $I(\gamma) = 0$ that start from $x_0 = 0$.

As $d = 1$, we drop all sub- and super-scripts $i \in \{1,\dots,d\}$ for the these two examples. 

\begin{example} \label{example:ehrenfest_diverging_jump_rates} 
Consider the one-dimensional Ehrenfest model with
\begin{equation*}
r_{n,+}(x) = \frac{2}{\sqrt{1-x}} \wedge n , \qquad r_{n,-}(x) = \frac{2}{\sqrt{1+x}} \wedge n.
\end{equation*}
Set $v_+(x) = \sqrt{1-x}$, $v_-(x) = \sqrt{1+x}$. By Dini's theorem, we have 
\begin{equation*}
\sup_{x \in [-1,1]} \left|\frac{1-x}{2} r_{n,+(x)} - v_+(x)\right| = 0, \qquad \sup_{x \in [-1,1]} \left|\frac{1+x}{2} r_{n,-(x)} - v_-(x)\right| = 0.
\end{equation*}
And additionally, conditions (a) and (b) of Theorem \ref{theorem:ldp_ehrenfest_dd} are satisfied, e.g. take $v_{+,1,\dagger}(x) = \sqrt{1-x}$, $v_{+,1,\ddagger}(x) = 1$. 
\end{example}

\begin{example} \label{example:ehrenfest_non_unique_solution}
Consider the one-dimensional Ehrenfest model with some rates $r_{n,+}$, $r_{n,-}$ and functions $v_+(x) > 0, v_-(x) > 0$ such that $\frac{1}{2}(1-x)r_{n,+}(x) \rightarrow v_+(x)$ and $\frac{1}{2}(1+x)r_{n,-}(x) \rightarrow v_-(x)$ uniformly in $x \in [-1,1]$.

Now suppose that there is a neighbourhood $U$ of $0$ on which $v_+,v_-$ have the form
\begin{equation*}
v_+(x) = \begin{cases}
1+ \sqrt{x} & x \geq 1, \\
1 & x < 1,
\end{cases} \qquad \qquad v_-(x) = 1.
\end{equation*} 
Consider the family of trajectories $t \mapsto \gamma_a(t)$, $a \geq 0$, defined by
\begin{equation*}
\gamma_a(t) := \begin{cases}
0 & \text{for } t \leq a, \\
(t-a)^2 & \text{for } t \geq a.
\end{cases}
\end{equation*}
Let $T > 0$ be small enough such that $\gamma_0(t) \in U$, and hence $\gamma_a(t) \in U$, for all $t \leq T$. A straightforward calculation yields $\int_0^T \cL(\gamma_a(t),\dot{\gamma}_a(t)) \dd t = 0$ for all $a \geq 0$. So we find multiple trajectories starting at $0$ that have zero Lagrangian cost.

Indeed, note that $\cL(x,v) = 0$ is equivalent to $v = H_p(x,0) = 2\left[v_+(x) - v_-(x) \right] = 2\sqrt{(x)}$. This yields that trajectories that have $0$ Lagrangian cost are the trajectories, at least in $U$, that solve
\begin{equation*}
\dot{\gamma}(t) = 2 \sqrt{\gamma(t)}
\end{equation*}
which is the well-known example of a differential equation that allows for multiple solutions.
\end{example}

We end with an example for Theorem \ref{theorem:ldp_mean_field_jump_process} and Proposition \ref{proposition:Htheorem} in the spirit of Example \ref{example:Ehrenfest_with_potential}. 

\begin{example}[Glauber dynamics for the Potts-model] \label{example:Gibbs_dynamics}
Fix some continuously differentiable function $V : \bR^d \rightarrow \bR$. Define the Gibbs measures
\begin{equation*}
\nu_n(\dd \sigma) := \frac{e^{-V(\mu_n(\sigma))}}{Z_n} P^{\otimes,n}(\dd \sigma)
\end{equation*}
on $\{1,\dots,d\}^n$, where $P^{\otimes,n}$ is the $n$-fold product measure of the uniform measure $P$ on $\{1,\dots,d\}$ and where $Z_n$ are normalizing constants.

Let $S(\mu \, | \, P)$ denote the relative entropy of $\mu \in \cP(\{1,\dots,d\})$ with respect to $P$:
\begin{equation*}
S(\mu \, | \, P) = \sum_a \log (d \mu(a)) \mu(a).
\end{equation*}
By Sanov's theorem and Varadhan's lemma, the empirical measures under the laws $\nu_n$ satisfy a large deviation principle with rate function $I_0(\mu) = S(\mu \, | \, P) + V(\mu)$.

Now fix some function $r : \{1,\dots,d\}\times\{1,\dots,d\} \rightarrow \bR^+$. Set
\begin{equation*}
r_n(a,b,\mu) = r(a,b)\exp\left\{- n 2^{-1}\left(V\left(\mu - n^{-1}\delta_a + n^{-1} \delta_b\right) - V(\mu)\right) \right\}.
\end{equation*}
As $n$ goes to infinity, we have uniform convergence of $\mu(a)r_n(a,b,\mu)$ to
\begin{equation*}
v(a,b,\mu) := \mu(a) r(a,b) \exp\left\{\frac{1}{2} \nabla_a V(\mu) - \frac{1}{2} \nabla_b V(\mu) \right\},
\end{equation*}
where $\nabla_a V(\mu)$ is the derivative of $V$ in the $a$-th coordinate. As in Example \ref{example:Ehrenfest_with_potential}, condition (b) of Theorem \ref{theorem:ldp_mean_field_jump_process} is satisfied by using the obvious decomposition.

\smallskip

By Proposition \ref{proposition:Htheorem}, we obtain that $S(\mu \, | \, P) + V(\mu)$ is Lyapunov function for 
\begin{equation*}
\dot{\mu}(a) = \sum_b \left[v(b,a,\mu) - v(a,b,\mu)\right] \qquad a \in \{1,\dots,d\}.
\end{equation*}
\end{example}

\subsection{Discussion and comparison to the existing literature} \label{section:comparison_to_results_in_literature}

We discuss our results in the context of the existing literature that cover our situation. Additionally, we consider a few cases where the large deviation principle(LDP) is proven for diffusion processes, because the proof techniques could possibly be applied in this setting.

\textbf{LDP: Approach via non-interacting systems, Varadhan's lemma and the contraction principle.}
In \cite{Le95,DPdH96,BoSu12}, the first step towards the LDP of the trajectory of some mean-field statistic of $n$ interacting particles is the LDP for non-interacting particles on some large product space obtained via Sanov's theorem. Varadhan's lemma then gives the LDP in this product space for interacting particles, after which the contraction principle gives the LDP on the desired trajectory space. In \cite{Le95,DPdH96}, the set-up is more general compared to ours in the sense that in \cite{Le95} the behaviour of the particles depends on their spatial location, and in \cite{DPdH96} the behaviour of a particle depends on some external random variable.

On the other hand, systems as in Example \ref{example:ehrenfest_diverging_jump_rates} fall outside of the conditions imposed in the three papers, if we disregard spatial dependence or external randomness. 

The approach via Varadhan's lemma, which needs control over the size of the perturbation, does not work, at least naively, for the situation where the jump rate for individual particles is diverging to $\infty$, or converging to $0$, if the mean is close to the boundary, see Remark \ref{remark:singular_rates}.

\smallskip

\textbf{LDP: Explicit control on the probabilities.} 
For another approach considering interacting spins that have a spatial location, see \cite{Com87}. The jump rates are taken to be explicit and the large deviation principle is proven via explicit control on the Radon-Nikodym derivatives. This method should in principle work also in the case of singular $v$. The approach via the generators $H_n$ in this paper, avoids arguments based on explicit control. This is an advantage for processes where the functions $r_n$ and $v$ are not very regular. Also in the classical Freidlin-Wentzell approach \cite{FW98} for dynamical systems with Gaussian noise the explicit form of the Radon-Nikodym derivatives is used to prove the LDP.

\smallskip

\textbf{LDP: Direct comparison to a process of independent particles.}
The main reference concerning large deviations for the trajectory of the empirical mean for interacting diffusion processes on $\bR^d$ is \cite{DG87}. In this paper, the large deviation principle is also first established for non-interacting particles. An explicit rate function is obtained by showing that the desired rate is in between the rate function obtained via Sanov's theorem and the contraction principle and the projective limit approach. The large deviation principle for interacting particles is then obtained via comparing the interacting process with a non-interacting process that has a suitably chosen drift. 
For related approaches, see \cite{Fe94b} for large deviations of interacting jump processes on $\bN$, where the interaction is unbounded and depends on the average location of the particles. See \cite{DjKa95} for mean field jump processes on $\bR^d$. 

Again, the comparison with non-interacting processes would fail in our setting due the singular interaction terms.

\smallskip

\textbf{LDP: Stochastic control.}
A more recent approach using stochastic control and weak convergence methods has proposed in the context of both jump and diffusion processes in \cite{BuDuMa11,BDF12}. A direct application of the results in \cite{BuDuMa11} fails for jump processes in the setting of singular behaviour at the boundary.

\smallskip

\textbf{LDP: Proof via operator convergence and the comparison principle.}
Regarding our approach based on the comparison principle, see \cite[Section 13.3]{FK06}, for an approach based on the comparison principle in the setting of \cite{DG87} and \cite{BDF12}. See \cite{DFL11} for an example of large deviations of a diffusion processes on $(0,\infty)$ with vanishing diffusion term with singular behaviour at the boundary. The methods to prove the comparison principle in Sections 9.2 and 9.3 in \cite{FK06} do not apply in our setting due to the different nature of our Hamiltonians. 

\smallskip

\textbf{LDP: Comparison of the approaches}
The method of obtaining exponential tightness in \cite{FK06}, and thus employed for this paper, is via density of the domain of the limiting generator $(H,\cD(H))$. Like in the theory of weak convergence, functions $f \in \cD(H)$ in the domain of the generator, and functions $f_n \in \cD(H_n)$ that converge to $f$ uniformly, can be used to bound the fluctuations in the Skorokhod space. This method is similar to the approaches taken in \cite{Co89,FW98,DG87}.

The approach using operator convergence is based on a projective limit theorem for the Skorokhod space. As we have exponential tightness on the Skorokhod space, it suffices to prove the large deviation principle for all finite dimensional distributions. This is done the convergence of the logarithmic moment generating functions for the finite dimensional distributions. The Markov property reduces this to the convergence of the  logarithmic moment generating function for time $0$ and convergence of the conditional moment generating functions, that form a semigroup $V_n(t)f(x) = \frac{1}{n} \log \bE[e^{nf(X_n(t))} \, | \, X_n(0) = x]$. Thus, the problem is reduced to proving convergence of semigroups $V_n(t)f \rightarrow V(t)f$. As in the theory of linear semigroups, this comes down to two steps. First one proves convergence of the generators $H_n \rightarrow H$. Then one shows that the limiting semigroup generates a semigroup. The verification of the comparison principle implies that the domain of the limiting operator is sufficiently large to pin down a limiting semigroup. 

This can be compared to the same problem for linear semigroups and the martingale problem. If the domain of a limiting linear generator is too small, multiple solutions to the martingale problem can be found, giving rise to multiple semigroups, see Chapter 12 in \cite{SV79} or Section 4.5 in \cite{EK86}.

The convergence of $V_n(t)f(x) \rightarrow V(t)f(x)$ uniformly in $x$ corresponds to having sufficient control on the Doob-h transforms corresponding to the change of measures
\begin{equation*}
\frac{\dd \PR_{n,x}^{f,t}}{\dd \PR_{n,x}}(X_n) = \exp\left\{nf(X_n(t))\right\},
\end{equation*}
where $\PR_{n,x}$ is the measure corresponding to the process $X_n$ started in $x$ at time $0$. An argument based on the projective limit theorem and control on the Doob h-transforms for independent particles is also used in \cite{DG87}, whereas the methods in \cite{Co89,FW98} are based on direct calculation of the probabilities being close to a target trajectories. 

\smallskip

\textbf{Large deviations for large excursions in large time.}
A notable second area of comparison is the study of large excursions in large time in the context of queuing systems, see e.g. \cite{DuIiSo90,DuEl95,DuAt99} and references therein. Here, it is shown that the rate functions themselves, varying in space and time, are solutions to a Hamilton-Jacobi equation. As in our setting, one of the main problems is the verification of the comparison principle. The notable difficulty in these papers is a discontinuity of the Hamiltonian at the boundary, but in their interior the rates are uniformly bounded away from infinity and zero.

\smallskip

\textbf{Lyapunov functions.}
In \cite{BuDuFiRa15a,BuDuFiRa15b}, Lyapunov functions are obtained for the McKean-Vlasov equation corresponding to interacting Markov processes in a setting similar to the setting of Theorem \ref{theorem:ldp_mean_field_jump_process}. Their discussion goes much beyond Proposition \ref{proposition:Htheorem}, which is perhaps best compared to Theorem 4.3 in \cite{BuDuFiRa15b}. However, the proof of Proposition \ref{proposition:Htheorem} is interesting in its own right, as it gives an intuitive explanation for finding a relative entropy as a Lyapunov functional and is not based on explicit calculations. In particular, the proof of Proposition \ref{proposition:Htheorem} in principle works for any setting where the path-space large deviation principle holds.

\section{Large deviation principle via an associated Hamilton-Jacobi equation} \label{section:LDP_via_HJequation}

In this section, we will summarize the main results of \cite{FK06}. Additionally, we will verify the main conditions of their results, except for the comparison principle of an associated Hamilton-Jacobi equation. This verification needs to be done for each individual model separately and this is the main contribution of this paper. We verify the comparison principle for our two models in Section \ref{section:viscosity_solutions}.

\subsection{Operator convergence}

We start by recalling some results from \cite{FK06}. Let $E_n$ and $E$ denote either of the spaces $E_{n,1}, E_1$ or $E_{n,2}, E_2$. Furthermore, denote by $C(E)$ the continuous functions on $E$ and by $C^1(E)$ the functions that are continuously differentiable on a neighbourhood of $E$ in $\bR^d$.

\smallskip

Assume that for each $n \in \bN$, we have a jump process $X_n$ on $E_n$, generated by a bounded infinitesimal generator $A_n$. For the two examples, this process is either $x_n$ or $\mu_n$. We denote by $\{S_n(t)\}_{t \geq 0}$ the transition semigroups $S_n(t)f(y) = \bE\left[f(X_n(t)) \, \middle| \, X_n(0) = y \right]$ on $C(E_n)$. 
Define for each $n$ the exponential semigroup
\begin{equation*}
V_n(t) f(y) := \frac{1}{n} \log S_n(t) e^{nf}(y) = \frac{1}{n} \log \bE\left[e^{nf(X_n(t))} \, \middle| \, X_n(0) = y \right].
\end{equation*}
Feng and Kurtz\cite{FK06} show that the existence of a strongly continuous limiting semigroup $\{V(t)\}_{t \geq 0}$ on $C(E)$ in the sense that for all $f \in C(E)$ and $T \geq 0$, we have
\begin{equation} \label{eqn:convergence_for_semigroups}
\lim_{n\rightarrow \infty} \sup_{t \leq T} \sup_{x \in E_n} \left| V(t)f(x) - V_n(t)f(x) \right| = 0,
\end{equation}
allows us to study to study the large deviation behaviour of the process $X_n$. We will consider this question from the point of view of the generators $H_n$ of $\{V_n(t)\}_{t \geq 0}$, where $H_n f$ is defined by the norm limit of $t^{-1} (V_n(t)f - f)$ as $t \downarrow 0$. Note that $H_n f = n^{-1} e^{-nf} A_n e^{nf}$, which for our first model yields
\begin{multline*}
H_n f(x) = \sum_{i = 1}^d \Bigg\{  \frac{1-x_i}{2}r_{n,+}^i(x) \left[\exp\left\{n \left(f\left(x + \frac{2}{n}e_i\right) - f(x)\right)\right\} - 1 \right] \\
+ \frac{1+x_i}{2}r_{n,-}^i(x)\left[\exp\left\{ n\left(f\left(x - \frac{2}{n}e_i\right) - f(x)\right)\right\} - 1 \right] \Bigg\}.
\end{multline*} 
For our second model, we have
\begin{equation*}
H_n f(\mu) =  \sum_{a,b = 1}^d \mu(a)r_n(a,b,\mu)\left[\exp\left\{n\left(f\left(\mu - n^{-1}\delta_a + n^{-1} \delta_b\right) - f(\mu)\right)\right\} - 1\right].
\end{equation*}

In particular, Feng and Kurtz show that, as in the theory of weak convergence of Markov processes, the existence of a limiting operator $(H,\cD(H))$, such that for all $f \in \cD(H)$ 
\begin{equation} \label{eqn:convergence_condition_Hamiltonians}
\lim_{n\rightarrow \infty} \sup_{x \in E_n} \left| Hf(x) - H_n f(x) \right| = 0,
\end{equation}
for which one can show that $(H,\cD(H))$ generates a semigroup $\{V(t)\}_{t \geq 0}$ on $C(E)$ via the Crandall-Liggett theorem, \cite{CL71}, then \eqref{eqn:convergence_for_semigroups} holds. 

\begin{lemma} \label{lemma:operator_convergence}
For either of our two models, assuming \eqref{eqn:main_convergence_condition_for_ldp_ehrenfest} or \eqref{eqn:property_of_jump_kernels}, we find that $H_n f \rightarrow Hf$, as in \eqref{eqn:convergence_condition_Hamiltonians} holds for $f \in C^1(E)$, where $Hf$ is given by $Hf(x) := H(x,\nabla f(x))$ and where $H(x,p)$ is defined in \eqref{eqn:def_Hamiltionian_Ehrenfest} or \eqref{eqn:def_Hamiltionian_jump}.
\end{lemma}

The proof of the lemma is straightforward using the assumptions and the fact that $f$ is continuously differentiable.

\smallskip

Thus, the problem is reduced to proving that $(H,C^1(E))$ generates a semigroup. The verification of the conditions of the Crandall-Liggett theorem is in general very hard, or even impossible. Two conditions need to be verified, the first is the \textit{dissipativity} of $H$, which can be checked via the positive maximum principle. The second condition is the \textit{range condition}: one needs to show that for $\lambda > 0$, the range of $(\bONE - \lambda H)$ is dense in $C(E)$. In other words, for $\lambda > 0$ and sufficiently many fixed $h \in C(E)$, we need to solve $f - \lambda H f = h$ with $f \in C^1(E)$. An alternative is to solve this equation in the \textit{viscosity sense}. If a viscosity solution exists and is unique, we denote it by $\tilde{R}(\lambda)h$. Using these solutions, we can extend the domain of the operator $(H,C^1(E))$ by adding all pairs of the form $(\tilde{R}(\lambda)h, \lambda^{-1}(\tilde{R}(\lambda)h - h))$ to the graph of $H$ to obtain an operator $\hat{H}$ that satisfies the conditions for the Crandall-Liggett theorem. This is part of the content of Theorem \ref{theorem_LDP_viscosity_FK} stated below. 

As a remark, note that any concept of weak solutions could be used to extend the operator. However, viscosity solutions are special in the sense that the extended operator remains dissipative.

The next result is a direct corollary of Theorem 6.14 in \cite{FK06}.

\begin{theorem} \label{theorem_LDP_viscosity_FK}
For either of our two models, assume that \eqref{eqn:main_convergence_condition_for_ldp_ehrenfest} or \eqref{eqn:property_of_jump_kernels} holds. Additionally, assume that the comparison principle is satisfied for \eqref{eqn:differential_equation_intro} for all $\lambda > 0$ and $h \in C(E)$.

Then, the operator
\begin{equation*}
\hat{H} := \bigcup_{\lambda > 0} \left\{ \left(\tilde{R}(\lambda)h, \lambda^{-1}(\tilde{R}(\lambda)h - h)\right) \, \middle| \, h \in C(E) \right\}
\end{equation*}
generates a semigroup $\{V(t)\}_{t \geq 0}$ as in the Crandall-Liggett theorem and we have \eqref{eqn:convergence_for_semigroups}.

Additionally, suppose that $\{X_n(0)\}$ satisfies the large deviation principle on $E$ with good rate function $I_0$. Then $X_n$ satisfies the large deviation principle on $D_E(\bR^+)$ with good rate function $I$ given by
\begin{equation*}
I(\gamma) = I_0(\gamma(0)) + \sup_{m} \sup_{0 = t_0 < t_1 < \dots < t_m} \sum_{k=1}^m I_{t_k - t_{k-1}}(\gamma(t_k) \, | \, \gamma(t_{k-1})),
\end{equation*}
where $I_s(y \, | \, x) := \sup_{f \in C(E)} f(y) - V(s)f(x)$.
\end{theorem}

Note that to prove Theorem 6.14 in \cite{FK06}, one needs to check that viscosity sub- and super-solutions to \eqref{eqn:differential_equation_intro} exist. Feng and Kurtz construct these sub- and super-solutions explicitly, using the approximating operators $H_n$, see the proof of Lemma 6.9 in \cite{FK06}.

\begin{proof}
We check the conditions for Theorem 6.14 in \cite{FK06}. In our models, the maps $\eta_n : E_n \rightarrow E$ are simply the embedding maps. Condition (a) is satisfied as all our generators $A_n$ are bounded. The conditions for convergence of the generators follow by Lemma \ref{lemma:operator_convergence}.
\end{proof}

The additional assumptions in Theorems \ref{theorem:ldp_ehrenfest_dd} and \ref{theorem:ldp_mean_field_jump_process} are there to make sure we are able to verify the comparison principle. This is the major contribution of the paper and will be carried out in Section \ref{section:viscosity_solutions}.

The final steps to obtain Theorems \ref{theorem:ldp_ehrenfest_dd} and \ref{theorem:ldp_mean_field_jump_process} are to obtain the rate function as the integral over a Lagrangian. Also this is based on results in Chapter 8 of \cite{FK06}.

\subsection{Variational semigroups}

In this section, we introduce the \textit{Nisio semigroup} $\mathbf{V}(t)$, of which we will show that it equals $V(t)$ on $C(E)$. This semigroup is given as a variational problem where one optimises a payoff $f(\gamma(t))$ that depends on the state $\gamma(t) \in E$,  but where a cost is paid that depends on the whole trajectory $\{\gamma(s)\}_{0 \leq s \leq t}$. The cost is accumulated over time and is given by a `Lagrangian'. Given the continuous and convex operator $Hf(x) = H(x,\nabla f(x))$, we define this Lagrangian by taking the Legendre-Fenchel transform:
\begin{equation*}
\cL(x,u) := \sup_{p \in \bR^d} \left\{\ip{p}{u} - H(x,p) \right\}.
\end{equation*}
As $p \mapsto H(x,p)$ is convex and continuous, it follows by the Fenchel - Moreau theorem that also
\begin{equation*}
Hf(x) = H(x,\nabla f(x)) = \sup_{u \in \bR^d} \left\{\ip{\nabla f(x)}{u} - \cL(x,u) \right\}.
\end{equation*}
Using $\cL$, we define the Nisio semigroup for measurable functions $f$ on $E$:
\begin{equation} \label{eqn:def_Nisio_semigroup}
\mathbf{V}(t)f(x) = \sup_{\substack{\gamma \in \cA\cC \\ \gamma(0) = x}} f(\gamma(t)) - \int_0^t \cL(\gamma(s),\dot{\gamma}(s)) \dd s.
\end{equation}

To be able to apply the results from Chapter 8 in \cite{FK06}, we need to verify Conditions 8.9 and 8.11 of \cite{FK06}.

\smallskip

For the semigroup to be well behaved, we need to verify Condition 8.9 in \cite{FK06}. In particular, this condition implies Proposition 8.19 in \cite{FK06} that ensures that the Nisio semigroup is in fact a semigroup on the upper semi-continuous functions that are bounded above. Additionally, it implies that all absolutely continuous trajectories up to time $T$, that have uniformly bounded Lagrangian cost, are a compact set in $D_E([0,T])$.

\begin{lemma} \label{lemma:condition_8_9}
For the Hamiltonians in \eqref{eqn:def_Hamiltionian_Ehrenfest} and \eqref{eqn:def_Hamiltionian_jump}, Condition 8.9 in \cite{FK06} is satisfied.
\end{lemma}

\begin{proof}
For (1),take $U = \bR^d$ and set $Af(x,v) = \ip{\nabla f(x)}{v}$. Considering Definition 8.1 in \cite{FK06}, if $\gamma \in \cA\cC$, then
\begin{equation*}
f(\gamma(t)) - f(\gamma(0)) = \int_0^t Af(\gamma(s),\dot{\gamma}(s)) \dd s
\end{equation*}
by definition of $A$. In Definition 8.1, however, relaxed controls are considered, i.e. instead of a fixed speed $\dot{\gamma}(s)$, one considers a measure $\lambda \in \cM(\bR^d \times \bR^+)$, such that $\lambda(\bR^d \times [0,t]) = t$ for all $ t \geq 0$ and
\begin{equation*}
f(\gamma(t)) - f(\gamma(0)) = \int_0^t Af(\gamma(s),v) \lambda(\dd v, \dd s).
\end{equation*}
These relaxed controls are then used to define the Nisio semigroup in equation (8.10). Note however, that by convexity of $H$ in the second coordinate, also $\cL$ is convex in the second coordinate. It follows that a deterministic control $\lambda(\dd v, \dd t) = \delta_{v(t)}(\dd v) \dd t$ is always the control with the smallest cost by Jensen's inequality. We conclude that we can restrict the definition (8.10) to curves in $\cA\cC$. This motivates our changed definition in equation \eqref{eqn:def_Nisio_semigroup}. 

For this paper, it suffices to set $\Gamma = E \times \bR^d$, so that (2) is satisfied. By compactness of $E$, (4) is clear. 

\smallskip

We are left to prove (3) and (5). For (3), note that $\cL$ is lower semi-continuous by construction. We also have to prove compactness of the level sets. By lower semi-continuity, it is suffices to show that the level sets $\{\cL \leq c\}$ are contained in a compact set.

Set $\cN := \cap_{x \in E} \left\{p \in \bR^d \, \middle| \, H(x,p) \leq 1 \right\}$. First, we show that $\cN$ has non-empty interior, i.e. there is some $\varepsilon > 0$ such that the open ball $B(0,\varepsilon)$ of radius $\varepsilon$ around $0$ is contained in $\cN$. Suppose not, then there exists $x_n$ and $p_n$ such that $p_n \rightarrow 0$ and for all $n$: $H(x_n,p_n) \in 1$. By compactness of $E$ and continuity of $H$, we find a value $H(x,0) = 1$, which contradicts our definitions of $H$, where $H(y,0) = 0$ for all $y \in E$.

Let $(x,v) \in \{\cL\leq c\}$, then
\begin{equation*}
\ip{p}{v} \leq \cL(x,v) + H(x,p) \leq c + 1
\end{equation*}
for all $p \in B(0,\varepsilon) \subseteq \cN$. It follows that $v$ is contained in some bounded ball in $\bR^d$. It follows that $\{\cL \leq c\}$ is contained in some compact set by the Heine-Borel theorem.

\smallskip

Finally, (5) can be proven as Lemma 10.21 in \cite{FK06} or Lemma 5.19 in \cite{Kr14a}
\end{proof}

The last property necessary for the equality of $V(t)f$ and $\mathbf{V}(t)f$ on $C(E)$ is the verification of Condition 8.11 in \cite{FK06}. This condition is key to proving that a variational resolvent, see equation (8.22), is a viscosity super-solution to \eqref{eqn:differential_equation_intro}. As the variational resolvent is also a sub-solution to \eqref{eqn:differential_equation_intro} by Young's inequality, the variational resolvent is a viscosity solution to this equation. If viscosity solutions are unique, this yields, after an approximation argument that $V(t) = \mathbf{V}(t)$. 

\begin{lemma} \label{lemma:condition_8_11}
Condition 8.11 in \cite{FK06} is satisfied. In other words, for all $g \in C^{1}(E)$ and $x_0 \in E$, there exists a trajectory $\gamma \in \cA\cC$ such that $\gamma(0) = x_0$ and for all $T \geq 0$:
\begin{equation} \label{eqn:lemma_optimising_trajectories}
\int_0^T Hg(\gamma(t)) \dd t = \int_0^T \ip{\nabla g(\gamma(t))}{\dot{\gamma}(t)} - \cL(\gamma(t),\dot{\gamma}(t)) \dd t.
\end{equation}
\end{lemma}

\begin{proof}
Fix $T > 0$, $g \in C^{1}(E)$ and $x_0 \in E$. We introduce a vector field $\mathbf{F}^g : E \rightarrow \bR^d$, by
\begin{align*}
\mathbf{F}^g(x) : = H_p(x,\nabla g(x)),
\end{align*}
where $H_p(x,p)$ is the vector of partial derivatives of $H$ in the second coordinate. Note that in our examples, $H$ is continuously differentiable in the $p$-coordinates. For example, for the $d=1$ case of Theorem \ref{theorem:ldp_ehrenfest_dd}, we obtain
\begin{equation*}
\mathbf{F}^g(x) : = 2v_+(x)e^{2\nabla g(x)} - 2v_-(x)e^{-2\nabla g(x)}.
\end{equation*}
As $\mathbf{F}^g$ is a continuous vector field, we can find a local solution $\gamma^g(t)$ in $E$ to the differential equation
\begin{equation*}
\begin{cases}
\dot{\gamma}(t) = \mathbf{F}^g(\gamma(t)), \\
\gamma(0) = x_0,
\end{cases}
\end{equation*}
by an extended version of Peano's theorem \cite{Cr72}. The result in \cite{Cr72} is local, however, the length of the interval on which the solution is constructed depends inversely on the norm of the vector field, see his equation (2). As our vector fields are globally bounded in size, we can iterate the construction in \cite{Cr72} to obtain a global existence result, such that $\dot{\gamma}^g(t) = \mathbf{F}^g(\gamma(t))$ for almost all times in $[0,\infty)$.

We conclude that on a subset of full measure of $[0,T]$ that
\begin{align*}
\cL(\gamma^g(t),\dot{\gamma}^g(t)) & = \cL(\gamma^g(t),\mathbf{F}^g(\gamma^g(t))) \\
& = \sup_{p \in \bR^d} \ip{p}{\mathbf{F}^g(\gamma^g(t))} - H(\gamma^g(t),p) \\
& = \sup_{p \in \bR^d} \ip{p}{H_p(\gamma^g(t), \nabla g(\gamma^g(t)))} - H(\gamma^g(t),p).
\end{align*}
By differentiating the final expression with respect to $p$, we find that the supremum is taken for $p = \nabla g(\gamma^g(t))$. In other words, we find
\begin{align*}
\cL(\gamma^g(t),\dot{\gamma}^g(t)) & = \ip{\nabla g(\gamma^g(t))}{H_p(\gamma^g(t), \nabla g(\gamma^g(t)))} - H(\gamma^g(t),\nabla g(\gamma^g(t))) \\
& = \ip{\nabla g(\gamma^g(t))}{\dot{\gamma}^g(t)} - Hg(\gamma^g(t)).
\end{align*}
By integrating over time, the zero set does not contribute to the integral, we find \eqref{eqn:lemma_optimising_trajectories}.
\end{proof}

The following result follows from Corollary 8.29 in \cite{FK06}.

\begin{theorem} \label{theorem:general_ldp}
For either of our two models, assume that \eqref{eqn:main_convergence_condition_for_ldp_ehrenfest} or \eqref{eqn:property_of_jump_kernels} holds. Assume that the comparison principle is satisfied for \eqref{eqn:differential_equation_intro} for all $\lambda > 0$ and $h \in C(E)$. Finally, suppose that $\{X_n(0)\}$ satisfies the large deviation principle on $E$ with good rate function $I_0$.

Then, we have $V(t)f = \mathbf{V}(t)f$ for all $f \in C(E)$ and $t \geq 0$. Also, $X_n$ satisfies the large deviation principle on $D_E(\bR^+)$ with good rate function $I$ given by
\begin{equation*}
I(\gamma) := \begin{cases}
I_0(\gamma(0)) + \int_0^\infty \cL(\gamma(s),\dot{\gamma}(s)) \dd s & \text{if } \gamma \in \cA\cC, \\
\infty & \text{if } \gamma \notin \cA\cC.
\end{cases}
\end{equation*}
\end{theorem}

\begin{proof}
We check the conditions for Corollary 8.29 in \cite{FK06}. Note that in our setting $H = \mathbf{H}$. Therefore, condition (a) of Corollary 8.29 is trivially satisfied. Furthermore, we have to check the conditions for Theorems 6.14 and 8.27. For the first theorem, these conditions were checkel already in the proof of our Theorem \ref{theorem_LDP_viscosity_FK}. For Theorem 8.27, we need to check Conditions 8.9, 8.10 and 8.11 in \cite{FK06}. As $H1 = 0$, Condition 8.10 follows from 8.11. 8.9 and 8.11 have been verified in Lemmas \ref{lemma:condition_8_9} and \ref{lemma:condition_8_11}.
\end{proof}

The last theorem shows us that we have Theorems \ref{theorem:ldp_ehrenfest_dd} and \ref{theorem:ldp_mean_field_jump_process} if we can verify the comparison principle, i.e. Theorems \ref{theorem:comparison_ehrenfest_ddim} and \ref{theorem:comparison_mean_field_jump}. This will be done in the section below.

\begin{proof}[Proof of Theorems \ref{theorem:ldp_ehrenfest_dd} and \ref{theorem:ldp_mean_field_jump_process}]
The comparison principles for equation \eqref{eqn:differential_equation_intro} are verified in Theorems \ref{theorem:comparison_ehrenfest_ddim} and \ref{theorem:comparison_mean_field_jump}. The two theorems now follow from Theorem \ref{theorem:general_ldp}.
\end{proof}

\begin{proof}[Proof of Proposition \ref{proposition:Htheorem}]
We give the proof for the system considered in Theorem \ref{theorem:ldp_ehrenfest_dd}. Fix $t \geq 0$ and some some starting point $x_0$. Let $x(t)$ be any solution of $\dot{x}(t) = \mathbf{F}(x(t))$ with $x(0) = x_0$. We show that $I_0(x(t)) \leq I_0(x_0)$.

\smallskip

Let $X_n(0)$ be distributed as $\nu_n$. Then it follows by Theorem \ref{theorem:ldp_ehrenfest_dd} that the large deviation principle holds for $\{X_n\}_{n \geq 0}$ on $D_E(\bR^+)$. 

\smallskip

As $\nu_n$ is invariant for the Markov process generated by $A_n$, also the sequence $\{X_n(t)\}_{n \geq 0}$ satisfies the large deviation principle on $E$ with good rate function $I_0$. Combining these two facts, the Contraction principle\cite[Theorem 4.2.1]{DZ98} yields
\begin{multline*}
I_0(x(t)) = \inf_{\gamma \in \cA\cC: \gamma(t) = x(t)} I_0(\gamma(0)) + \int_0^t \cL(\gamma(s),\dot{\gamma}(s)) \dd s \\
\leq I_0(x(0)) + \int_0^t \cL(x(s),\dot{x}(s)) \dd s  = I_0(x(0)).
\end{multline*}
Note that $\cL(x(s),\dot{x}(s)) = 0$ for all $s$ as was shown in Example \ref{example:ehrenfest_non_unique_solution}.
\end{proof}

\section{The comparison principle} \label{section:viscosity_solutions}

We proceed with checking the comparison principle for equations of the type $f(x) - \lambda B(x,\nabla f(x)) - h(x) = 0$. In other words, for subsolutions $u$ and supersolutions $v$ we need to check that $u \leq v$. We start with some known results. First of all, we give the main tool to construct sequences $x_\alpha$ and $y_\alpha$ that converge to a maximising point $z \in E$ such that $u(z) - v(z) = \sup_{z'\in E} u(z') - v(z')$. This result can be found for example as Proposition 3.7 in \cite{CIL92}.

\begin{lemma}\label{lemma:doubling_lemma}
Let $E$ be a compact subset of $\bR^d$, let $u$ be upper semi-continuous, $v$ lower semi-continuous and let $\Psi : E^2 \rightarrow \bR^+$ be a lower semi-continuous function such that $\Psi(x,y) = 0$ implies $x = y$. For $\alpha > 0$, let $x_\alpha,y_\alpha \in E$ such that
\begin{equation*}
u(x_\alpha) - v(y_\alpha) - \alpha \Psi(x_\alpha,y_\alpha) = \sup_{x,y \in E} \left\{u(x) - v(y) - \alpha \Psi(x,y) \right\}.
\end{equation*}
Then the following hold
\begin{enumerate}[(i)]
\item  $\lim_{\alpha \rightarrow \infty} \alpha \Psi(x_\alpha,y_\alpha) = 0$.
\item All limit points of $(x_\alpha,y_\alpha)$ are of the form $(z,z)$ and for these limit points we have $u(z) - v(z) = \sup_{x \in E} \left\{u(x) - v(x) \right\}$.
\end{enumerate}
\end{lemma}

We say that $\Psi : E^2 \rightarrow \bR^+$ is a \textit{good distance function} if $\Psi(x,y) = 0$ implies $x = y$, it is continuously differentiable in both components and if $(\nabla \Psi(\cdot,y))(x) = - (\nabla \Psi(x,\cdot))(y)$ for all $x,y \in E$. The next two results can be found as Lemma 9.3 in \cite{FK06}. We will give the proofs of these results for completeness.

\begin{proposition} \label{proposition:comparison_conditions_on_H}
Let $(B,\cD(B))$ be an operator such that $\cD(B) = C^{1}(E)$ of the form $Bf(x) = B(x,\nabla f(x))$. Let $u$ be a subsolution and $v$ a supersolution to $f(x) - \lambda B(x,\nabla f(x)) - h(x) = 0$, for some $\lambda > 0$ and $h \in C(E)$. Let $\Psi$ be a good distance function and let $x_\alpha,y_\alpha$ satisfy
\begin{equation*}
u(x_\alpha) - v(y_\alpha) - \alpha \Psi(x_\alpha,y_\alpha) = \sup_{x,y \in E} \left\{u(x) - v(y) - \alpha \Psi(x,y) \right\}.
\end{equation*}
Suppose that
\begin{equation*}
\liminf_{\alpha \rightarrow \infty} B\left(x_\alpha,\alpha (\nabla \Psi(\cdot,y_\alpha))(x_\alpha)\right) - B\left(y_\alpha,\alpha (\nabla \Psi(\cdot,y_\alpha))(x_\alpha)\right) \leq 0,
\end{equation*}
then $u \leq v$. In other words, $f(x) - \lambda B(x,\nabla f(x)) - h(x) = 0$ satisfies the comparison principle.
\end{proposition}

\begin{proof}
Fix $\lambda >0$ and $h \in C(E)$. Let $u$ be a subsolution and $v$ a supersolution to 
\begin{equation} \label{eqn:proof_comp_equation}
f(x) - \lambda B(x,\nabla f(x)) - h(x) = 0.
\end{equation}
We argue by contradiction and assume that $\delta := \sup_{x \in E} u(x) - v(x) > 0$. For $\alpha > 0$, let $x_\alpha,y_\alpha$ be such that
\begin{equation*}
u(x_\alpha) - v(y_\alpha) - \alpha \Psi(x_\alpha,y_\alpha) = \sup_{x,y \in E} \left\{u(x) - v(y) - \alpha \Psi(x,y) \right\}.
\end{equation*}
Thus Lemma \ref{lemma:doubling_lemma} yields $\alpha \Psi(x_\alpha,y_\alpha) \rightarrow 0$ and for any limit point $z$ of the sequence $x_\alpha$, we have $u(z) - v(z) = \sup_{x \in E} u(x) - v(x)  = \delta > 0$. It follows that for $\alpha$ large enough, $u(x_\alpha) - v(y_\alpha) \geq \frac{1}{2}\delta$.

\smallskip

For every $\alpha > 0$, the map $\Phi^1_\alpha(x) := v(y_\alpha) + \alpha\Psi(x,y_\alpha)$ is in $C^{1}(E)$ and $u(x) - \Phi^1_\alpha(x)$ has a maximum at $x_\alpha$. On the other hand, $\Phi^2_\alpha(y) := u(x_\alpha) - \alpha \Psi(x_\alpha,y)$ is also in $C^{1}(E)$ and $v(y) - \Phi^2_\alpha(y)$ has a minimum at $y_\alpha$. As $u$ is a sub- and $v$ a super solution to \eqref{eqn:proof_comp_equation}, we have 
\begin{align*} 
\frac{u(x_\alpha) - h(x_\alpha)}{\lambda} & \leq H(x_\alpha,\alpha(\nabla\Psi(\cdot,y_\alpha))(x_\alpha)) \\
\frac{v(y_\alpha) - h(y_\alpha))}{\lambda} & \geq H(y_\alpha,-\alpha(\nabla\Psi(x_\alpha,\cdot))(y_\alpha)) \\
& = H(y_\alpha,\alpha(\nabla\Psi(\cdot,y_\alpha))(x_\alpha))
\end{align*}
where the last equality follows as $\Psi$ is a good distance function. It follows that for $\alpha$ large enough, we have
\begin{align}
0 & < \frac{\delta}{2\lambda}  \leq \frac{u(x_\alpha) - v(y_\alpha)}{\lambda} \label{eqn:comp_proof_contradicting_assumption} \\
& = \frac{u(x_\alpha) - h(x_\alpha)}{\lambda} - \frac{v(y_\alpha) - h(y_\alpha)}{\lambda} + \frac{1}{\lambda}\left(h(x_\alpha) - h(y_\alpha)\right) \notag \\
& \leq H(x_\alpha,\alpha(\nabla\Psi(\cdot,y_\alpha))(x_\alpha)) - H(y_\alpha,\alpha(\nabla\Psi(\cdot,y_\alpha))(x_\alpha)) + \frac{1}{\lambda}\left(h(x_\alpha) - h(y_\alpha)\right) \notag 
\end{align}
As $h$ is continuous, we obtain $\lim_{\alpha \rightarrow \infty} h(x_\alpha) - h(y_\alpha) = 0$. Together with the assumption of the proposition, we find that the $\liminf_{\alpha \rightarrow \infty} c_\alpha \leq 0$ which contradicts by  \eqref{eqn:comp_proof_contradicting_assumption} that $\delta > 0$.
\end{proof}

The next lemma gives additional control on the sequences $x_\alpha,y_\alpha$.

\begin{lemma} \label{lemma:control_on_H}
Let $(B,\cD(B))$ be an operator such that $\cD(B) = C^{1}(E)$ of the form $Bf(x) = B(x,\nabla f(x))$. Let $u$ be a subsolution and $v$ a supersolution to $f(x) - \lambda B(x,\nabla f(x)) - h(x) = 0$, for some $\alpha > 0$ and $h \in C(E)$. Let $\Psi$ be a good distance function and let $x_\alpha,y_\alpha$ satisfy
\begin{equation*}
u(x_\alpha) - v(y_\alpha) - \alpha \Psi(x_\alpha,y_\alpha) = \sup_{x,y \in E} \left\{u(x) - v(y) - \alpha \Psi(x,y) \right\}.
\end{equation*}
Then we have that
\begin{equation} \label{eqn:control_on_H}
\sup_\alpha B\left(y_\alpha,\alpha (\nabla \Psi(\cdot,y_\alpha))(x_\alpha)\right) < \infty.
\end{equation}
\end{lemma}

\begin{proof}
Fix $\lambda > 0$, $h \in C(E)$ and let $u$ and $v$ be sub- and super-solutions to $f(x) - \lambda B(x,f(x)) - h(x) = 0$. Let $\Psi$ be a good distance function and let $x_\alpha,y_\alpha$ satisfy
\begin{equation*}
u(x_\alpha) - v(y_\alpha) - \alpha \Psi(x_\alpha,y_\alpha) = \sup_{x,y \in E} \left\{u(x) - v(y) - \alpha \Psi(x,y) \right\}.
\end{equation*}
As $y_\alpha$ is such that
\begin{equation*}
v(y_\alpha) - \left(u(x_\alpha) - \Psi(x_\alpha,y_\alpha)\right) = \inf_y v(y) - \left(u(x_\alpha) - \Psi(x_\alpha,y)\right), 
\end{equation*}
and $v$ is a super-solution, we obtain
\begin{equation*}
B\left(y_\alpha,-\alpha (\nabla \Psi(x_\alpha,\cdot))(y_\alpha)\right) \leq \frac{v(y_\alpha) - h(y_\alpha)}{\lambda}
\end{equation*}
As $\Phi$ is a good distance function, we have $- (\nabla \Psi(x_\alpha,\cdot))(y_\alpha) = (\nabla \Psi(\cdot,y_\alpha))(x_\alpha)$. The boundedness of $v$ now implies
\begin{equation*} 
\sup_\alpha B\left(y_\alpha,\alpha (\nabla \Psi(\cdot,y_\alpha))(x_\alpha)\right) \leq \frac{1}{\alpha} \left(v(y_\alpha) - h(y_\alpha) \right) \leq \vn{v - h} < \infty.
\end{equation*}
\end{proof}

\subsection{One-dimensional Ehrenfest model} \label{subsection:Ehrenfest_model_1d}

To single out the important aspects of the proof of the comparison principle for equation \eqref{eqn:differential_equation_intro}, we start by proving it for the $d=1$ case of Theorem \ref{theorem:ldp_ehrenfest_dd}.

\begin{proposition} \label{proposition:comparison_ehrenfest_1d}
Let $E = [-1,1]$ and let 
\begin{equation*}
H(x,p) = v_+(x) \left[e^{2p} -1\right] + v_-(x) \left[e^{-2p} - 1\right],
\end{equation*}
where $v_+, v_-$ are continuous and satisfy the following properties: 
\begin{enumerate}[(a)]
\item $v_+(x) = 0$ for all $x$ or $v_+$ satisfies the following properties:
\begin{enumerate}[(i)]
\item $v_+(x) > 0$ for $x \neq 1$.
\item $v_+(1) = 0$ and there exists a neighbourhood $U_{1}$ of $1$ on which there exists a decomposition $v_+(x) = v_{+,\dagger}(x)v_{+,\ddagger}(x)$ such that $v_{+,\dagger}$ is decreasing and where $v_{+,\ddagger}$ is continuous and satisfies $v_{+,\ddagger}(1) \neq 0$.
\end{enumerate}
\item $v_-(x) = 0$ for all $x$ or $v_-$ satisfies the following properties:
\begin{enumerate}[(i)]
\item $v_-(x) > 0$ for $x \neq -1$.
\item $v_+(-1) = 0$ and there exists a neighbourhood $U_{-1}$ of $1$ on which there exists a decomposition $v_-(x) = v_{-,\dagger}(x)v_{-,\ddagger}(x)$ such that $v_{-,\dagger}$ is increasing and where $v_{-,\ddagger}$ is continuous and satisfies $v_{-,\ddagger}(-1) \neq 0$.
\end{enumerate}
\end{enumerate}
Let $\lambda > 0$ and $h \in C(E)$. Then the comparison principle holds for
 $f(x) - \lambda H(x,\nabla f(x)) - h(x) = 0$.
\end{proposition}

\begin{proof}
Fix $\lambda > 0$, $h \in C(E)$ and pick a sub- and super-solutions $u$ and $v$ to $f(x) - \lambda H(x,\nabla f(x)) - h(x) = 0$. We check the condition for Proposition \ref{proposition:comparison_conditions_on_H}. We take the good distance function $\Psi(x,y) = 2^{-1} (x-y)^2$ and let $x_\alpha,y_\alpha$ satisfy
\begin{equation*}
u(x_\alpha) - v(y_\alpha) - \frac{\alpha}{2} |x_\alpha-y_\alpha|^2 = \sup_{x,y \in E} \left\{u(x) - v(y) - \frac{\alpha}{2} |x-y|^2 \right\}.
\end{equation*}
We need to prove that
\begin{equation} \label{eqn:comp_proof_1d_basic_inequality}
\liminf_{\alpha \rightarrow \infty} H(x_\alpha,\alpha(x_\alpha-y_\alpha)) - H(y_\alpha,\alpha(x_\alpha-y_\alpha)) \leq 0.
\end{equation}
By Lemma \ref{lemma:doubling_lemma}, we know that $\alpha|x_\alpha - y_\alpha|^2 \rightarrow 0$ as $\alpha \rightarrow \infty$ and any limit point of $x_\alpha, y_\alpha$ is of the form $(z,z)$ for some $z$ such that $u(z) - v(z) = \max_{z' \in E} u(z') - v(z')$. 
Restrict $\alpha$ to the sequence $\alpha \in \bN$ and extract a subsequence, which we will also denote by $\alpha$, such that $\alpha \rightarrow \infty$ $x_\alpha$ and $y_\alpha$ converge to some $z$. The rest of the proof depends on whether $z = -1, z = 1$ or $z \in (-1,1)$.

\smallskip

First suppose that $z \in (-1,1)$. By Lemma \ref{lemma:control_on_H}, we have
\begin{equation*}
\sup_\alpha v_+(y_\alpha) \left[e^{2\alpha(x_\alpha-y_\alpha)} - 1 \right] + v_-(y_\alpha) \left[e^{-2\alpha(x_\alpha-y_\alpha)} - 1 \right] < \infty.
\end{equation*}
As $e^c -1 > -1$, we see that the $\limsup$ of both terms of the sum individually are bounded as well. Using that $y_\alpha \rightarrow z \in (-1,1)$, and the fact that $v_+,v_-$ are bounded away from $0$ on a closed interval around $z$, we obtain from the first term that $\sup_\alpha \alpha(x_\alpha - y_\alpha) < \infty$ and from the second that $\sup_\alpha \alpha(y_\alpha - x_\alpha) < \infty$. We conclude that $\alpha(x_\alpha - y_\alpha)$ is a bounded sequence. Therefore, there exists a subsequence $\alpha(k)$ such that $\alpha(k)(x_{\alpha(k)} - y_{\alpha(k)})$ converges to some $p_0$. We find that
\begin{align*}
& \liminf_{\alpha \rightarrow \infty} H(x_\alpha,\alpha(x_\alpha-y_\alpha)) - H(y_\alpha,\alpha(x_\alpha-y_\alpha)) \\
& \leq \lim_{k \rightarrow \infty} H(x_{\alpha(k)},\alpha(x_{\alpha(k)}-y_{\alpha(k)}) - H(y_{\alpha(k)},\alpha(x_{\alpha(k)}-y_{\alpha(k)})) \\
& = H(z,p_0) - H(z,p_0) = 0 
\end{align*}

We proceed with the proof in the case that $x_\alpha,y_\alpha \rightarrow z = -1$. The case where $z = 1$ is proven similarly. Again by Lemma \ref{lemma:control_on_H}, we obtain the bounds
\begin{equation}
\sup_\alpha v_+(y_\alpha) \left[e^{2\alpha(x_\alpha-y_\alpha)} - 1 \right]  < \infty,  \qquad \sup_\alpha v_-(y_\alpha) \left[e^{-2\alpha(x_\alpha-y_\alpha)} - 1 \right]  < \infty.  \label{eqn:comp_proof_1d_second_sup_bound}
\end{equation}

As $v_+$ is bounded away from $0$ near $-1$, we obtain by the left hand bound that $\sup_\alpha \alpha(x_\alpha - y_\alpha) < \infty$. As in the proof above, it follows that if $\alpha|x_\alpha - y_\alpha|$ is bounded, we are done. This leaves the case where there exists a subsequence of $\alpha$, denoted by $\alpha(k)$, such that $\alpha(k)(y_{\alpha(k)} - x_{\alpha(k)}) \rightarrow \infty$. Then clearly, $e^{2\alpha(k)(x_{\alpha(k)} - y_{\alpha(k)})}- 1$ is bounded and contains a converging subsequence. We obtain as in the proof where $z \in (-1,1)$ that
\begin{align*}
& \liminf_{\alpha \rightarrow \infty} H(x_\alpha,\alpha(x_\alpha-y_\alpha)) - H(y_\alpha,\alpha(x_\alpha-y_\alpha)) \\
& \quad = \liminf_{\alpha \rightarrow \infty} \left[v_+(x_\alpha) - v_+(y_\alpha) \right]\left[e^{2\alpha(x_\alpha-y_\alpha)} - 1 \right] \\
& \qquad \qquad \qquad + \left[v_-(x_\alpha) - v_-(y_\alpha) \right]\left[e^{2\alpha(y_\alpha-x_\alpha)} - 1 \right] \\
& \quad \leq \liminf_{k \rightarrow \infty} \left[v_-(x_{\alpha(k)}) - v_-(y_{\alpha(k)}) \right]\left[e^{2\alpha(k)(y_{\alpha(k)}-x_{\alpha(k)})} - 1 \right].
\end{align*}
Note that as $\alpha(k)(y_{\alpha(k)} - x_{\alpha(k)}) \rightarrow \infty$, we have $y_{\alpha(k)} > x_{\alpha(k)} \geq  0$. Also for $k$ sufficiently large, $y_{\alpha(k)}, x_{\alpha(k)} \in U_{-1}$. It follows that $v_-(y_{\alpha(k)}) > 0$, which allows us to write
\begin{align*}
& \left[v_-(x_{\alpha(k)}) - v_-(y_{\alpha(k)}) \right]\left[e^{2\alpha(k)(y_{\alpha(k)}-x_{\alpha(k)})} - 1 \right] \\
& = \left[\frac{v_{-,\dagger}(x_{\alpha(k)})}{v_{-,\dagger}(y_{\alpha(k)})}\frac{v_{-,\ddagger}(x_{\alpha(k)})}{v_{-,\ddagger}(y_{\alpha(k)})} - 1 \right]v_-(y_{\alpha(k)}) \left[e^{2\alpha(k)(y_{\alpha(k)}-x_{\alpha(k)})} - 1 \right].
\end{align*}
By the bound in \eqref{eqn:comp_proof_1d_second_sup_bound}, and the obvious lower bound, we see that the non-negative sequence
\begin{equation*}
u_k := v_-(y_{\alpha(k)}) \left[e^{2\alpha(k)(y_{\alpha(k)}-x_{\alpha(k)})} - 1 \right]
\end{equation*}
contains a converging subsequence $u_{k'} \rightarrow c$. As $y_{\alpha(k)} > x_{\alpha(k)}$ and $v_{-,\dagger}$ is increasing:
\begin{multline*}
\limsup_k \frac{v_{-,\dagger}(x_{\alpha(k)})}{v_{-,\dagger}(y_{\alpha(k)})}\frac{v_{-,\ddagger}(x_{\alpha(k)})}{v_{-,\ddagger}(y_{\alpha(k)})} \\
 \leq \left(\limsup_k \frac{v_{-,\dagger}(x_{\alpha(k)})}{v_{-,\dagger}(y_{\alpha(k)})}\right)\left(\lim_k  \frac{v_{-,\ddagger}(x_{\alpha(k)})}{v_{-,\ddagger}(y_{\alpha(k)})} \right) \leq  \frac{v_{-,\ddagger}(-1)}{v_{-,\ddagger}(-1)} = 1.
\end{multline*}
As a consequence, we obtain
\begin{multline*}
\liminf_{k} \left[\frac{v_{-}(x_{\alpha(k)})}{v_{-}(y_{\alpha(k)})} - 1 \right]v_-(y_{\alpha(k)}) \left[e^{2\alpha(k)(y_{\alpha(k)}-x_{\alpha(k)})} - 1 \right] \\
\leq \left(\limsup_k\left[\frac{v_{-,\dagger}(x_{\alpha(k)})}{v_{-,\dagger}(y_{\alpha(k)})}\frac{v_{-,\ddagger}(x_{\alpha(k)})}{v_{-,\ddagger}(y_{\alpha(k)})} - 1 \right] \right) \left(\liminf_{k'}  u_{k'} \right) \leq 0.
\end{multline*}
This concludes the proof of \eqref{eqn:comp_proof_1d_basic_inequality} for the case that $z = -1$.
\end{proof}

\subsection{Multi-dimensional Ehrenfest model}

\begin{comment}
\begin{proposition} \label{proposition:comparison_ehrenfest_ddim}
Consider 
\begin{equation*}
H(x, p) = \sum_i v_+^i(x)\left[e^{2p_i} - 1 \right]  + v_-^i(x) \left[e^{-2p_i} - 1 \right].
\end{equation*}
for $x \in [-1,1]^d$ Let $v_+^i, v_-^i$ satisfy the properties stated in Theorem \ref{theorem:ldp_ehrenfest_dd}. Then, for any $\lambda > 0$ and $h \in C(E)$ the comparison principle holds for $f(x) - \lambda H(x,\nabla f(x)) - h(x) = 0$.
\end{proposition}
\end{comment}

\begin{proof}[Proof of Theorem \ref{theorem:comparison_ehrenfest_ddim}]
Let $u$ be a subsolution and $v$ a supersolution to $f(x) - \lambda H(x,\nabla f(x)) - h(x) = 0$. As in the proof of Proposition \ref{proposition:comparison_ehrenfest_1d}, we check the condition for Proposition \ref{proposition:comparison_conditions_on_H}. Again, for $\alpha \in \bN$ let $x_\alpha,y_\alpha$ satisfy
\begin{equation*}
u(x_\alpha) - v(y_\alpha) - \frac{\alpha}{2} |x_\alpha-y_\alpha|^2 = \sup_{x,y \in E} \left\{u(x) - v(y) - \frac{\alpha}{2} |x-y|^2 \right\}.
\end{equation*}
and without loss of generality let $z$ be such that $x_\alpha,y_\alpha \rightarrow z$.

\smallskip

Denote with $x_{\alpha,i}$ and $y_{\alpha,i}$ the $i$-th coordinate of $x_\alpha$ respectively $y_\alpha$. We prove
\begin{multline*}
\liminf_{\alpha \rightarrow \infty} H(x_\alpha,\alpha(x_\alpha-y_\alpha)) - H(y_\alpha,\alpha(x_\alpha-y_\alpha)) \\
= \liminf_{\alpha \rightarrow \infty} \sum_i \Bigg\{ \left[v_+^i(x_\alpha) - v_+^i(y_\alpha)\right]\left[e^{\alpha(x_{\alpha,i} - y_{\alpha,i})} - 1 \right] \\
+ \left[v_-^i(x_\alpha) - v_-^i(y_\alpha)\right]\left[e^{\alpha(y_{\alpha,i} - x_{\alpha,i})} - 1 \right] \Bigg\} \leq 0, 
\end{multline*}
by constructing a subsequence $\alpha(n) \rightarrow \infty$ such that the first term in the sum converges to $0$. From this sequence, we find a subsequence such that the second term converges to zero, and so on.

\smallskip

Therefore, we will assume that we have a sequence $\alpha(n) \rightarrow \infty$ for which the first $i-1$ terms of the difference of the two Hamiltonians vanishes and prove that we can find a subsequence for which the $i$-th term 
\begin{multline} \label{eqn:Hamiltonian_Ehrenfest_comparioson_ith_term}
\left[v_+^i(x_\alpha) - v_+^i(y_\alpha)\right]\left[e^{\alpha(x_{\alpha,i} - y_{\alpha,i})} - 1 \right] \\
+ \left[v_-^i(x_\alpha) - v_-^i(y_\alpha)\right]\left[e^{\alpha(y_{\alpha,i} - x_{\alpha,i})} - 1 \right]
\end{multline}
vanishes. This follows directly as in the proof of Proposition \ref{proposition:comparison_ehrenfest_1d}, arguing depending on the situation $z_i \in (-1,1)$, $z_i = -1$ or $z_i = -1$.
\end{proof}

\subsection{Mean field Markov jump processes}

\begin{comment}
\begin{proposition} \label{proposition:comparison_mean_field_jump} 
Consider on $E = E_2$, the Hamiltonian
\begin{equation*}
H(x,p) = \sum_{a,b} v(a,b,\mu)\left[e^{p_b - p_a} -1\right].
\end{equation*}
Suppose that for all $a,b$ the continuous map $\mu \mapsto v(a,b,\mu)$ satisfies the properties stated in Theorem \ref{theorem:ldp_mean_field_jump_process}. Let $\lambda > 0$ and $h \in C(E)$, then the comparison principle holds for $f(\mu) - \lambda H(\mu,\nabla f(\mu)) - h(\mu) = 0$.
\end{proposition}
\end{comment}

The proof of Theorem \ref{theorem:comparison_mean_field_jump} follows along the lines of the proofs of Proposition \ref{proposition:comparison_ehrenfest_1d} and Theorem \ref{theorem:comparison_ehrenfest_ddim}. The proof however needs one important adaptation because of the appearance of the difference $p_b - p_a$ in the exponents of the Hamiltonian.

Naively copying the proofs using the distance function $\Psi(\mu,\nu) = \frac{1}{2} \sum_{a} (\mu(a) - \nu(a))^2$ one obtains by Lemma \ref{lemma:control_on_H} , for suitable sequences $\mu_\alpha$ and $\nu_\alpha$, that
\begin{equation*}
\sup_\alpha v(a,b,\nu_\alpha)\left[e^{\alpha\left(\left(\mu_\alpha(b) - \nu_\alpha(b)\right) - \left(\mu_\alpha(a) - \nu_\alpha(a)\right)\right)} - 1 \right] < \infty.
\end{equation*}
One sees that the control on the sequences $\alpha(\nu_\alpha(a) - \mu_\alpha(a))$ obtained from this bound is not very good, due to the compensating term $\alpha(\mu_\alpha(b) - \nu_\alpha(b))$. 

\smallskip

The proof can be suitably adapted using a different distance function. For $x \in \bR$, let $x^- := x \wedge 0$ and $x^+ = x \vee 0$. Define $\Psi(\mu,\nu) = \frac{1}{2} \sum_{a} ((\mu(a) - \nu(a))^-)^2 = \frac{1}{2} \sum_{a} ((\nu(a) - \mu(a))^+)^2 $. Clearly, $\Psi$ is differentiable in both components and satisfies $(\nabla \Psi(\cdot,\nu))(\mu) = - (\nabla \Psi(\mu,\cdot))(\nu)$. Finally, using the fact that $\sum_i \mu(i) = \sum_i \nu(i) = 1$, we find that $\Psi(\mu,\nu) = 0$ implies that $\mu = \nu$. We conclude that $\Psi$ is a good distance function.

The bound obtained from Lemma \ref{lemma:control_on_H} using this $\Psi$ yields 
\begin{equation*}
\sup_\alpha v(a,b,\nu_\alpha)\left[e^{\alpha\left(\left(\mu_\alpha(b) - \nu_\alpha(b)\right)^- - \left(\mu_\alpha(a) - \nu_\alpha(a)\right)^-\right)} - 1 \right] < \infty.
\end{equation*}
We see that if $\left(\mu_\alpha(b) - \nu_\alpha(b)\right)^- - \left(\mu_\alpha(a) - \nu_\alpha(a)\right)^- \rightarrow \infty$ it must be because $\alpha(\nu_\alpha(a) - \mu_\alpha(a)) \rightarrow \infty$. This puts us in the position to use the techniques from the previous proofs. 

\begin{proof}[Proof of Theorem \ref{theorem:comparison_mean_field_jump}]
Set $\Psi(\mu,\nu) = \frac{1}{2} \sum_{a} ((\mu(a) - \nu(a))^-)^2$, as above. We already noted that $\Psi$ is a good distance function.

\smallskip

Let $u$ be a subsolution and $v$ be a supersolution to $f(\mu) - \lambda H(\mu,\nabla f(\mu)) - h(\mu) = 0$. For $\alpha \in \bN$, pick $\mu_\alpha$ and $\nu_\alpha$ such that
\begin{equation*}
u(\mu_\alpha) - v(\nu_\alpha) - \alpha\Psi(\mu_\alpha,\nu_\alpha) = \sup_{\mu,\nu \in E} \left\{u(\mu) - v(\nu) - \alpha\Psi(\mu,\nu) \right\}
\end{equation*}
Furthermore, assume without loss of generality that $\mu_\alpha,\nu_\alpha \rightarrow z$ for some $z$ such that $u(z) - v(z) = \sup_{z'\in E} u(z') - v(z')$. By Proposition \ref{proposition:comparison_conditions_on_H}, we need to bound
\begin{align}
& H(\mu_\alpha,\alpha (\nabla\Phi(\cdot,\nu_\alpha))(\mu_\alpha)) - H(\nu_\alpha,\alpha (\nabla\Phi(\mu_\alpha,\cdot))(\mu_\alpha)) \notag \\
& = \sum_{a,b} \left[v(a,b,\mu_\alpha) - v(a,b,\nu_\alpha) \right]\left[e^{\alpha\left(\left(\mu_\alpha(b) - \nu_\alpha(b)\right)^- - \left(\mu_\alpha(a) - \nu_\alpha(a)\right)^-\right)} - 1\right]. \label{eqn:proof_comparison_Potts_sum}
\end{align}

As in the proof of Theorem \ref{theorem:comparison_ehrenfest_ddim}, we will show that each term in the sum above can be bounded above by $0$ separately. So pick some ordering of the ordered pairs $(i,j)$, $i,j \in \{1,\dots,n\}$ and assume that we have some sequence $\alpha$ such that the $\liminf_{\alpha \rightarrow \infty}$ of the first $k$ terms in equation \eqref{eqn:proof_comparison_Potts_sum} are bounded above by $0$. Suppose that $(i,j)$ is the pair corresponding to the $k+1$-th term of the sum in \eqref{eqn:proof_comparison_Potts_sum}.

\smallskip

Clearly, if $v(i,j,\pi) = 0$ for all $\pi$ then we are done. Therefore, we assume that $v(i,j,\pi) \neq 0$ for all $\pi$ such that $\pi(i) > 0$.

In the case that $\mu_\alpha, \nu_\alpha \rightarrow \pi^*$, where $\pi^*(i) > 0$, we know by Lemma \ref{lemma:control_on_H}, using that $v(i,j,\cdot)$ is bounded away from $0$ on a neighbourhood of $\pi^*$, that 
\begin{equation*}
\sup_\alpha e^{\alpha\left(\left(\mu_\alpha(j) - \nu_\alpha(j)\right)^- - \left(\mu_\alpha(i) - \nu_\alpha(i)\right)^-\right)} - 1 < \infty.
\end{equation*}
Picking a subsequence $\alpha(n)$ such that this term above converges and using that $\pi \rightarrow v(i,j,\pi)$ is uniformly continuous, we see
\begin{align*}
& \liminf_{\alpha \rightarrow \infty} \left[v(i,j,\mu_\alpha) - v(i,j,\nu_\alpha) \right]\left[e^{\alpha\left(\left(\mu_\alpha(j) - \nu_\alpha(j)\right)^- - \left(\mu_\alpha(i) - \nu_\alpha(i)\right)^-\right)} - 1\right] \\
& \quad = \lim_{n \rightarrow \infty} \left[v(i,j,\mu_{\alpha(n)}) - v(i,j,\nu_{\alpha(n)}) \right] \times \\
& \qquad \qquad \qquad \qquad \qquad \left[e^{{\alpha(n)}\left(\left(\mu_{\alpha(n)}(j) - \nu_{\alpha(n)}(j)\right)^- - \left(\mu_{\alpha(n)}(i) - \nu_{\alpha(n)}(i)\right)^-\right)} - 1\right] \\
& \quad = 0
\end{align*}

For the second case, suppose that $\mu_\alpha(i),\nu_\alpha(i) \rightarrow 0$. By Lemma \ref{lemma:control_on_H}, we get
\begin{equation} \label{eqn:proof_comparison_jump_sup_bound_on_exp}
\sup_\alpha v(i,j,\nu_\alpha) \left[e^{\alpha\left(\left(\mu_\alpha(j) - \nu_\alpha(j)\right)^- - \left(\mu_\alpha(i) - \nu_\alpha(i)\right)^-\right)} - 1\right] < \infty.
\end{equation}
First of all, if $\sup_\alpha \alpha\left(\left(\mu_\alpha(j) - \nu_\alpha(j)\right)^- - \left(\mu_\alpha(i) - \nu_\alpha(i)\right)^-\right) < \infty$, then the argument given above also takes care of this situation. So suppose that this supremum is infinite. Clearly, the contribution $\left(\mu_\alpha(j) - \nu_\alpha(j)\right)^-$ is negative, which implies that $\sup_\alpha \alpha\left(\nu_\alpha(i) -\mu_\alpha(i)\right)^+ = \infty$. This means that we can assume without loss of generality that
\begin{equation} \label{eqn:comparison_jump_assumption_measures}
\alpha\left(\nu_\alpha(i) - \mu_\alpha(i)\right) \rightarrow \infty, \qquad \nu_\alpha(i) > \mu_\alpha(i). 
\end{equation}
We rewrite the term $a = i$, $b = j$ in equation \eqref{eqn:proof_comparison_Potts_sum} as
\begin{equation*}
\left[\frac{v(i,j,\mu_\alpha)}{v(i,j,\nu_\alpha)} - 1 \right]v(i,j,\nu_\alpha) \left[e^{\alpha\left(\left(\mu_\alpha(j) - \nu_\alpha(j)\right)^- - \left(\mu_\alpha(i) - \nu_\alpha(i)\right)^-\right)} - 1\right].
\end{equation*}
The right hand side is bounded above by \eqref{eqn:proof_comparison_jump_sup_bound_on_exp} and bounded below by $-1$, so we take a subsequence of $\alpha$, also denoted by $\alpha$, such that the right hand side converges. Also note that for $\alpha$ large enough the right hand side is non-negative. Therefore, it suffices to show that
\begin{equation*}
\liminf_{\alpha \rightarrow \infty} \frac{v(i,j,\mu_\alpha)}{v(i,j,\nu_\alpha)} \leq 1,
\end{equation*}
which follows as in the proof of Proposition \ref{proposition:comparison_ehrenfest_1d}.
\end{proof}

\smallskip

\textbf{Acknowledgement}
The author thanks Frank Redig and Christian Maes for helpful discussions. Additionally, the author thanks anonymous referees for suggestions that improved the text. The author is supported by The Netherlands Organisation for Scientific Research (NWO), grant number 600.065.130.12N109.

\bibliographystyle{plain} 
\bibliography{../KraaijBib}{}

\end{document}